\documentclass[twoside]{amsart}
\usepackage{mathrsfs}
\usepackage{amsfonts}
\usepackage{amssymb}
\usepackage{amssymb}
\usepackage{amssymb}
\usepackage{amssymb}
\usepackage{amssymb}
\usepackage{amssymb}
\usepackage{amssymb}
\usepackage{amssymb}
\usepackage{amsmath}
\usepackage{amsthm}
\usepackage[all]{xy}

\usepackage{lastpage}

\RequirePackage{amsmath} \RequirePackage{amssymb}
\usepackage{amscd,latexsym,amsthm,amsfonts,amssymb,amsmath,amsxtra}
\usepackage
[
colorlinks=true, 
linkcolor=blue,
urlcolor=blue,
bookmarks=true,
bookmarksopen=true,
citecolor=blue,
]
{hyperref}

\newcommand{\bpm}{\begin{pmatrix}}
\newcommand{\epm}{\end{pmatrix}}
\newcommand{\bsm}{\begin{smallmatrix}}
\newcommand{\esm}{\end{smallmatrix}}
\newcommand{\bspm}{\left(\begin{smallmatrix}}
\newcommand{\espm}{\end{smallmatrix}\right)}
\newcommand{\bbm}{\begin{bmatrix}}
\newcommand{\ebm}{\end{bmatrix}}

    \newcommand{\BA}{{\mathbb {A}}} 
    \newcommand{\BC}{{\mathbb {C}}}

     \newcommand{\BZ}{{\mathbb {Z}}}

     \newcommand{\sH}{{\mathscr {H}}}

     \newcommand{\CP}{{\mathcal {P}}}
     \newcommand{\CR}{{\mathcal {R}}}
    \newcommand{\CS}{{\mathcal {S}}} 
     
    \newcommand{\CW}{{\mathcal {W}}}

     \newcommand{\fo}{{\mathfrak{o}}}

    \newcommand{\RG}{{ {G}}}

     \newcommand{\RP}{{\mathrm {P}}}

    \newcommand{\lenth}{{\mathrm {\lenth}}}

     \newcommand{\GL}{{\mathrm{GL}}}
    \newcommand{\Hom}{{\mathrm{Hom}}} 
    \newcommand{\Ind}{{\mathrm{Ind}}}

    \renewcommand{\Re}{{\mathrm{Re}}}

 \newcommand{\SL}{{\mathrm{SL}}}

\newcommand{\vol}{{\mathrm{vol}}}

 \newcommand{\Sp}{{\mathrm{Sp}}} \newcommand{\diag}{{\mathrm{diag}}} \newcommand{\pr}{{\mathrm{pr}}}

    \newcommand{\bx}{{\bf {x}}}       
\newcommand{\bG}{{{G}}}  \newcommand{\bW}{{\bf {W}}}

    \newcommand{\wt}{\widetilde}  
    
    \newcommand{\pair}[1]{\langle {#1} \rangle}
    \newcommand{\wpair}[1]{\left\{{#1}\right\}}

    \newcommand{\ov}{\overline}
    
    \newcommand{\incl}{\hookrightarrow}
    
     \newcommand{\ra}{\rightarrow}

    \theoremstyle{plain}
    
       \newtheorem*{theorem*}{Theorem}

    \newtheorem{thm}{Theorem}[section] 
    \newtheorem{lem}[thm]{Lemma}  \newtheorem{prop}[thm]{Proposition}
     
\newtheorem{rmk}[thm]{Remark}

    \numberwithin{equation}{section}

\usepackage[top=1in, bottom=1in, left=1.25in, right=1.25in]{geometry}

\title{On a Rankin-Selberg integral of the $L$-function for $\wt \SL_2\times \GL_2$}
\author{Qing Zhang}

\address{Department of Mathematical Sciences, KAIST, 291, Daehak-ro, Yuseong-gu, Daejeon, 34141, Korea}
\email{qingzhang0@gmail.com}

\begin{document}

\maketitle

\begin{abstract}
    We present a Rankin-Selberg integral on the exceptional group $G_2$ which represents the $L$-function for generic cuspidal representations of $\wt \SL_2\times \GL_2$. As an application, we show that certain Fourier-Jacobi type periods on $G_2$ are non-vanishing.
\end{abstract}

\section{Introduction}

Let $F$ be a global field with the ring of adeles $\BA$. We assume that the characteristics of $F$ is not 2. We present in this paper a Shimura type integral on the exceptional group $G_2$ which represents the $L$-function 
$$L(s,\wt\pi\times (\chi\otimes \tau))L(s,\wt\pi\otimes (\chi\otimes\omega_\tau)),$$
where $\wt\pi$ is an irreducible genuine cuspidal representation of $\wt \SL_2(\BA) $, $\tau$ is an irreducible generic cuspidal representation of $\GL_2(\BA)$ and $\chi$ is the quadratic character of $F^\times\backslash\BA^\times$ defined by $\chi(a)=\prod_v (a_v,-1)_{F_v}$, where $a=(a_v)_v \in \BA^\times$ and $(~,~)_{F_v}$ is the Hilbert symbol on $F_v$. 

To give more details about the integral, we introduce some notations. The group $G_2$ has two simple roots and we label the short root by $\alpha$ and the long root by $\beta$. Let $P=MV$ (resp. $P'=M'V'$) be the maximal parabolic subgroup of $G_2$ such that the root space of $\beta$ is in the Levi $M$ (resp. the root space of $\alpha$ is in the Levi $M'$). The Levi subgroups $M$ and $M'$ are isomorphic to $\GL_2$. Let $J$ be the subgroup of $P$ which is isomorphic to $\SL_2\ltimes V$. Let $\wt \SL_2(\BA)$ be the metaplectic double cover of $ \SL_2(\BA)$. There is a Weil representation $\omega_\psi$ of $\wt \SL_2(\BA)$ for a nontrivial additive character $\psi$ of $F\backslash \BA$. Let $\wt \theta_\phi$ be a corresponding theta series associated with a function $\phi\in \CS(\BA)$. Let $\tau$ be an irreducible cuspidal automorphic representations of $\GL_2(\BA)$. For $f_s\in \Ind_{P'(\BA)}^{G_2(\BA)}(\tau\otimes \delta_{P'}^s)$, we can form an Eisenstein series $E(g,f_s)$ on $G_2(\BA)$.  Let $\wt \pi$ be an irreducible genuine cuspidal automorphic forms of $\wt \SL_2(\BA)$. For a cusp form $\wt \varphi\in \wt \pi$, we consider the integral
$$I(\wt \varphi, \phi,f_s)=\int_{\SL_2(F)\backslash \SL_2(\BA)}\int_{V(F)\backslash V(\BA)}\wt \varphi(g)\wt\theta_\phi(vg)E(vg,f_s)dvdg.$$
Our main result is the following
\begin{thm}\label{thm: main 1}
The above integral is absolutely convergent for $\Re(s)\gg 0$ and can be meromorphically continued to all $s\in \BC$. When $\Re(s)\gg 0$, the integral $I(\wt\varphi,\phi,f_s)$ is Eulerian. Moreover, at an unramified place $v$, the local integral represents
the $L$-function
$$\frac{L(3s-1,\wt\pi_v\times (\chi_v\otimes\tau_v))L(6s-5/2,\wt\pi_v\otimes (\chi_v\otimes\omega_{\tau_v}))}{L(3s-1/2,\tau_v)L(6s-2,\omega_{\tau_v})L(9s-7/2,\tau_v\otimes\omega_{\tau_v})}. $$ 
Here $\chi_v$ is the unramified nontrivial quadratic character of $F_v^\times$.
\end{thm}
This is Theorem \ref{thm: eulerian} and Proposition \ref{prop: unramified calculation}.  We remark that Ginzburg-Rallis-Soudry gave integral representations for $L$-functions of generic cuspidal representations of $ \wt{\Sp}_{2n}\times \GL_m$ in \cite{GRS98} using symplectic groups. It is still interesting to have different integral representations. As an application of Theorem \ref{thm: main 1}, we show that if $Wd_\psi(\wt \pi)=\chi\otimes \tau$, then a Shimura type period with respect to $\wt \pi$ and the residue of Eisenstein series on $G_2$ is non-vanishing, where $Wd_\psi$ is the Shimura-Waldspurger lift. It is an interesting theme in number theory to investigate the relations between poles of $L$-functions and non-vanishing of automorphic periods. There are many examples of this kind relations. See \cite{JS, Gi93, GRS97} for some examples. The non-vanishing results of automorphic periods have many interesting applications in automorphic forms. We expect the non-vanishing period in our case would be useful on problems related to the residue spectrum of $G_2$.

 There are several known Rankin-Selberg integrals on $G_2$ which represents different $L$-functions and have many applications, see \cite{Gi91, Gi93, Gi95} for example. The integral $I(\wt \varphi,\phi,f_s)$ can be viewed as a dual integral of the standard $G_2$ $L$-function integral in \cite{Gi93} in the following sense. The integral $I(\wt \varphi,\phi,f_s)$ is an integral of a triple product of a cusp form on $\wt\SL_2(\BA)$, a theta series and an Eisenstein series on $G_2(\BA)$, while the integral in \cite{Gi93} is an integral of a triple product of a cusp form on $G_2(\BA)$, a theta series and an Eisenstein series on $\wt \SL_2(\BA)$. The integral in \cite{Gi95} is also in a similar pattern, which is an integral of a triple product of a cusp form on $\SL_2(\BA)$, a theta series and an Eisenstein series on a cover of $G_2(\BA)$. The results presented here were known for D. Ginzburg. But we still think that it might be useful to write up the details. 

\section*{Acknowledgements}
I would like to thank D. Ginzburg for helpful communications and pointing out the reference \cite{Gi95}. The debt of this paper to Ginzburg's papers \cite{Gi93, Gi95} should be evident for the readers. I also would like to thank Joseph Hundley and Baiying Liu for useful discussions. I appreciate Jim Cogdell and Clifton Cunningham for encouragement and support. I also would like to thank the anonymous referee for his/her careful reading and useful suggestions. This work is supported by a fellowship from Pacific Institute for Mathematical Sciences (PIMS) and NSFC grant 11801577.

\section{The group \texorpdfstring{$\RG_2$}{Lg}}

\subsection{Roots and Weyl group for \texorpdfstring{$\RG_2$}{}}
Let $\bG_2$ be the split algebraic reductive group of type $\bG_2$ (defined over $\BZ$). The group $\bG_2$ has two simple roots, the short root $\alpha$ and the long root $\beta$. The set of the positive roots is $\Sigma^+=\wpair{\alpha, \beta, \alpha+\beta, 2\alpha+\beta, 3\alpha+\beta, 3\alpha+2\beta}$.  Let $(~,~)$ be the inner product in the root system and $\pair{~,~}$ be the pair defined by $\pair{\gamma_1,\gamma_2}=\frac{2(\gamma_1,\gamma_2)}{(\gamma_2,\gamma_2)}$. For the root space $\bG_2$, we have the relations:
$$ \pair{\alpha,\beta}=-1, \pair{\beta, \alpha}=-3.$$
For a root $\gamma$, let $s_\gamma$ be the reflection defined by $\gamma$, i.e., $s_\gamma(\gamma')=\gamma'-\pair{\gamma',\gamma}\gamma$. We have the relation
$$s_\alpha(\beta)=3\alpha+\beta, s_\beta(\alpha)=\alpha+\beta.$$
The Weyl group $\bW=\bW(\bG_2)$ of $\bG_2$ has 12 elements, which is explicitly given by 
$$\bW=\wpair{1, s_\alpha, s_\beta, s_\alpha s_\beta, s_\beta s_\alpha, s_\alpha s_\beta s_\alpha, s_\beta s_\alpha s_\beta, (s_\alpha s_\beta)^2, (s_\beta s_\alpha)^2, s_\beta(s_\alpha s_\beta)^2, s_\alpha (s_\beta s_\alpha)^2, (s_\alpha s_\beta)^3}.$$

\iffalse
We have the following relations
$$\begin{array}{rccrcc}
s_\alpha s_\beta(\alpha)&=&2\alpha+\beta, &\quad s_\alpha s_\beta(\beta)&=&-(3\alpha+\beta), \\
s_\beta s_\alpha(\alpha)&=&-(\alpha+\beta), &\quad s_\beta s_\alpha(\beta)&=&3\alpha+2\beta,\\
s_\alpha s_\beta s_\alpha(\alpha) &=&-(2\alpha+\beta),&\quad s_\alpha s_\beta s_\alpha (\beta)&=&3\alpha+2\beta,\\
s_\beta s_\alpha s_\beta(\alpha)&=&2\alpha+\beta,& \quad s_\beta s_\alpha s_\beta(\beta)&=&-(3\alpha+2\beta),\\
(s_\alpha s_\beta)^2(\alpha)&=&\alpha+\beta, &\quad (s_\alpha s_\beta)^2(\beta)&=&-(3\alpha +2\beta),\\
(s_\beta s_\alpha)^2(\alpha)&=&-(2\alpha+\beta), & \quad (s_\beta s_\alpha)^2(\beta)&=&3\alpha+\beta,\\
s_\alpha(s_\beta s_\alpha)^2(\alpha)&=&-(\alpha+\beta), &\quad s_\alpha(s_\beta s_\alpha)^2(\beta)&=&\beta,\\ 
s_\beta(s_\alpha s_\beta)^2(\alpha)&=&\alpha, & \quad s_\beta(s_\alpha s_\beta)^2(\beta)&=&-(3\alpha+\beta),\\
(s_\alpha s_\beta)^3(\alpha)&=&-\alpha,& \quad (s_\alpha s_\beta)^3 (\beta)&=&-\beta.
\end{array}$$
From the above relations, we can check that
\begin{equation}
\begin{split}\label{eq1.1}
s_{\alpha+\beta}&=s_\beta s_\alpha s_\beta,\\
 s_{2\alpha+\beta}&=s_\alpha (s_\beta s_\alpha)^2,   \\
s_{3\alpha+\beta}&= s_\alpha s_\beta s_\alpha, \\
 s_{3\alpha+2\beta}&=s_\beta(s_\alpha s_\beta)^2.
\end{split}
\end{equation}

For simplicity, we denote $$w_1=s_\beta (s_\alpha s_\beta)^2, w_2=s_\alpha (s_\beta s_\alpha)^2, \textrm{ and }w_0=(s_\alpha s_\beta)^3.$$
\fi

 For a root $\gamma$, let $U_\gamma \subset G$ be the root space of $\gamma$, and  let $\bx_\gamma: F\ra U_\gamma$ be a fixed isomorphism which satisfies various Chevalley relations, see Chapter 3 of \cite{St}. Among other things, $\bx_{\gamma}$ satisfies the following commutator relations: 
\begin{equation}
\begin{split}\label{eq:commutator}
[\bx_\alpha(x), \bx_\beta(y)]&=  \bx_{\alpha+\beta}(-xy)\bx_{2\alpha+\beta}(-x^2 y) \bx_{3\alpha+\beta}( x^3 y)\bx_{3\alpha+2\beta}(-2x^3 y^2)\\
[\bx_\alpha(x), \bx_{\alpha+\beta}(y)]&= \bx_{2\alpha+\beta}( -2xy)\bx_{3\alpha+\beta}(3x^2 y) \bx_{3\alpha+2\beta}(3xy^2)  \\
[\bx_\alpha(x), \bx_{2\alpha+\beta}(y)]&=\bx_{3\alpha+\beta}( 3xy)  \\
[\bx_\beta(x), \bx_{3\alpha+\beta}(y)]&=\bx_{3\alpha+2\beta}(xy)  \\
[\bx_{\alpha+\beta}(x), \bx_{2\alpha+\beta}(y)]&=\bx_{3\alpha+2\beta}(3xy). 
\end{split}
\end{equation}
For all the other pairs of positive roots $\gamma_1, \gamma_2$, we have $[\bx_{\gamma_1}(x), \bx_{\gamma_2}(y)]=1$. Here $[g_1,g_2]=g_1^{-1}g_2^{-1}g_1g_2$ for $g_1,g_2\in G_2$. For these commutator relationships, see \cite{Re}.

 Following \cite{St}, we denote $w_\gamma(t)= \bx_{\gamma}(t)\bx_{-\gamma}(-t^{-1})\bx_{\gamma}(t)$ and $w_\gamma=w_\gamma(1)$. Note that $w_\gamma$ is a representative of $s_\gamma$. Let $h_\gamma(t)=w_\gamma(t) w_\gamma^{-1}$. Let $T$ be the subgroup of $G$ which consists of elements of the form $h_\alpha(t_1)h_\beta(t_2), t_1,t_2\in T$ and $U$ be the subgroup of $G_2$ generated by $U_\gamma$ for all $\gamma\in \Sigma^+$. Let $B=TU$, which is a Borel subgroup of $G_2$. 
 
For $t_1,t_2\in \mathbb{G}_m,$ denote $h(t_1,t_2)=h_\alpha(t_1t_2)h_\beta(t_1^2t_2)$. From the Chevalley relation $h_{\gamma_1}(t)\bx_{\gamma_2}(r) h_{\gamma_1}(t)^{-1}=\bx_{\gamma_2}(t^{\pair{\gamma_2,\gamma_1}} r)$ (see \cite[Lemma 20, (c)]{St}), we can check the following relations
\begin{equation}
\begin{split}\label{eq:torusactiononroots}
h^{-1}(t_1,t_2)\bx_\alpha(r)h(t_1,t_2)&=\bx_\alpha (t_2^{-1} r), \\
 h^{-1}(t_1,t_2)\bx_\beta(r)h(t_1,t_2)&=\bx_\beta(t_1^{-1}t_2 r)\\
h^{-1}(t_1,t_2)\bx_{\alpha+\beta}(r)  h(t_1,t_2)&=\bx_{\alpha+\beta}(t_1^{-1}r),\\
 h^{-1}(t_1,t_2)\bx_{2\alpha+\beta}(r) h(t_1,t_2)&= \bx_{2\alpha+\beta}(t_1^{-1}t_2^{-1}r) \\
h^{-1}(t_1,t_2)\bx_{3\alpha+\beta}(r)h(t_1,t_2)&=\bx_{3\alpha+\beta}(t_1^{-1}t_2^{-2}r),\\
 h^{-1}(t_1,t_2) \bx_{3\alpha+2\beta}(r) h(t_1,t_2)&=\bx_{3\alpha+2\beta}(t_1^{-2} t_2^{-1} r) .
\end{split}
\end{equation}
Thus the notation $h(a,b)$ agrees with that of \cite{Gi93}.

One can also check that 
\begin{equation*} w_\alpha h(t_1,t_2) w_\alpha^{-1}= h(t_1 t_2, t_2^{-1}), \quad w_\beta h(t_1,t_2) w_\beta^{-1}= h(t_2,t_1).\end{equation*}

\iffalse
We also need the Chevalley relation $w_{\gamma_1} \bx_{\gamma_2}(r) w_{\gamma_1}^{-1}=\bx_{w_{\gamma_1}(\gamma_2)}(c(\gamma_1,\gamma_2) r)$ (see \cite[Lemma 20, (b)]{St}), where $c(\gamma_1,\gamma_2)\in \wpair{\pm1}$. We have $c(\gamma_1,\gamma_2)=c(\gamma_1,-\gamma_2)$. The number $c(\gamma_1,\gamma_2)$ can be explicitly determined using matrix calculations:
\begin{align}\label{eq1.5} c(\alpha, \alpha)=c(\alpha, 2\alpha+\beta)=c(\alpha,3\alpha+\beta)=-1, &\quad  c(\alpha, \beta)=c(\alpha, \alpha+\beta)=c(3\alpha+2\beta)=1,\\
c(\beta,\beta)=c(\beta,\alpha+\beta)=c(\beta,3\alpha+2\beta)=-1, &\quad c(\beta,\alpha)=c(\beta,2\alpha+\beta)=c(\beta,3\alpha+\beta)=1.\nonumber
\end{align}
\fi
\subsection{Subgroups} Let $F$ be a field and denote $G=G_2(F)$. The group $G$ has two proper parabolic subgroups. Let $P=M\ltimes V$ be the parabolic subgroup of $G$ such that $U_\beta\subset M\cong \GL_2$. Thus the unipotent subgroup $V$ is consisting of root spaces of $\alpha, \alpha+\beta, 2\alpha+\beta, 3\alpha+\beta,3\alpha+2\beta$, and a typical element of $V$ is of the form 
$$ \bx_{\alpha}(r_1)\bx_{\alpha+\beta}(r_2)\bx_{2\alpha+\beta}(r_3)\bx_{3\alpha+\beta}(r_4)\bx_{3\alpha+2\beta}(r_5), r_i\in F.$$
To ease the notation, we will write the above element as $[r_1,r_2,r_3,r_4,r_5]$. Denote by $J$ the following subgroup of $P$  
$$J=\SL_2(F)\ltimes V.$$

  Let $ V_1$ (resp. $Z$) be the subgroup of $V$ which consists root spaces of $3\alpha+\beta$ and $3\alpha+2\beta$ (resp. $2\alpha+\beta, 3\alpha+\beta$ and $3\alpha+2\beta$). Note that $P$ and hence $J$ normalizes $V_1$ and $Z$. We will always view $\SL_2(F)$ as a subgroup of $G$ via the inclusion $\SL_2(F)\subset M$. Denote by $A_{\SL_2}$, $N_{\SL_2}$ and $B_{\SL_2}$ the standard torus, the upper triangular unipotent subgroup and the upper triangular Borel subgroup of $\SL_2(F)$. Note that the torus element $h(a,b)$ can be identified with 
$$\begin{pmatrix}a&\\ &b \end{pmatrix}\in \GL_2(F)\cong M,$$
and thus $A_{\SL_2}=\wpair{h(a,a^{-1})|a\in F^\times}$ and $B_{\SL_2}=A_{\SL_2}\ltimes U_\beta$.

Let $P'=M'V'$ be the other maximal parabolic subgroups of $G$ with $U_\alpha$ in the Levi subgroup $M'$. The Levi $M'$ is isomorphic to $\GL_2(F)$, and from relations in (\ref{eq:torusactiononroots}), one can check that one isomorphism $M'\cong \GL_2(F)$ can be determined by 
\begin{equation*}
\begin{split}\bx_\alpha(r)&\mapsto \begin{pmatrix}1& r\\ &1 \end{pmatrix},\\
h(a,b)&\mapsto \begin{pmatrix}ab&\\ &a \end{pmatrix}.
\end{split}\end{equation*}
In particular, we see that $h(a,1)\in T\subset M'$ can be identified with $\diag(a,a).$ Let $\delta_{P'}$ be the modulus character of $P'$. One can check that $\delta_{P'}(m')=|\det(m')|^3$ for $m'\in M'$, where $\det(m')$ can be computed using the above isomorphism $M'\cong \GL_2(F)$.

\subsection{Weil representation of \texorpdfstring{$\widetilde \SL_2(\BA)\ltimes V(\BA)$}{Lg}} In this subsection, we assume that $F$ is a global field and $\BA$ is its ring of adeles.
In $\SL_2(F)$, we denote $t(a)=\diag(a,a^{-1}),a\in F^\times$ and 
$$n(b)=\begin{pmatrix}1& b\\ &1 \end{pmatrix}, b\in F.$$
 Denote $w^1=\begin{pmatrix}&1\\ -1& \end{pmatrix}$, which represents the unique nontrivial Weyl element of $\SL_2(F)$. Under the embedding $\SL_2(F)\subset M\subset G$, the element $w^1$ can be identified with $w_\beta$.

Let $\widetilde \SL_2(\BA)$ be the metaplectic double cover of $\SL_2(\BA)$. Then we have an exact sequence
$$0\ra \mu_2\ra \widetilde \SL_2(\BA) \ra \SL_2(\BA)\ra 0,$$
where $\mu_2=\wpair{\pm 1}$.

We will identify $\SL_2(\BA)$ with the symplectic group of $\BA^2$ with symplectic structure defined by 
$$\pair{(x_1,y_1),(x_2,y_2)}=-2x_1y_2+2x_2y_1.$$

Let $\sH(\BA)$ be the Heisenberg group of the symplectic space $(\BA^2, \pair{~,~})$, i.e., $\sH(\BA)=\BA^3$ with group law
$$(x_1,y_1,z_1)(x_2,y_2,z_2)=(x_1+x_2,y_1+y_2, z_1+z_2-x_1y_2+y_1 x_2).$$

Let $\SL_2(\BA)$ act on $\sH(\BA)$ from the right side by $$(x_1, y_1, z_1).g=((x_1,y_1)g, z_1), g\in \SL_2(\BA),$$
where $(x_1,y_1)g$ is the usual matrix multiplication.

We then can form the semi-direct product $\SL_2(\BA)\ltimes \sH(\BA),$ where the product is defined by $$(g_1, h_1)(g_2,h_2)=(g_1g_2, (h_1.g_2 ) h_2), g_i\in \SL_2(\BA), h_i\in \sH(\BA), i=1,2.$$
Let $\psi$ be a nontrivial additive character of $F\backslash \BA$. Then there is a Weil representation $\omega_\psi$ of $\widetilde \SL_2(\BA)\ltimes \sH(\BA)$. The space of $\omega_\psi$ is $\CS(\BA)$, the Bruhat-Schwartz functions on $\BA$.

For $\phi\in \CS(\BA)$, we have the well-know formulas:
\begin{align*}
(\omega_\psi(n(b))\phi)(x)&=\psi(bx^2)\phi(x), b\in \BA \\
(\omega_{\psi}((r_1,r_2,r_3))\phi)(x)&=\psi\left( r_3-2xr_2-r_1r_2\right) \phi(x+r_1), (r_1,r_2,r_3)\in \sH(\BA),
\end{align*}
The above formulas could be found in \cite{Ku}.\\

Recall that for $r_1,r_2,r_3,r_4,r_5\in \BA$, the notation $[r_1,r_2,r_3,r_4,r_5]$ is an abbreviation of  $$\bx_\alpha(r_1)\bx_{\alpha+\beta}(r_2)\bx_{2\alpha+\beta}(r_3)\bx_{3\alpha+\beta}(r_4) \bx_{3\alpha+2\beta}(r_5)\in V(\BA).$$
Define a map
$\pr: V(\BA)\ra \sH(\BA)$
$$\pr([r_1,r_2,r_3,r_4,r_5])=(r_1,r_2,r_3-r_1 r_2) . $$
From the commutator relation (\ref{eq:commutator}), we can check that $\pr$ is a group homomorphism and defines an exact sequence
$$0\ra  V_1(\BA) \ra V(\BA)\ra \sH(\BA)\ra 0.$$
Recall that $V_1$ is the subgroup of $V$ which is generated by the root space of $3\alpha+\beta, 3\alpha+2\beta$. Note that there is a typo in the formula of the projection map $\pr$ in \cite[p.316]{Gi93}.

For $g=\begin{pmatrix}a & b\\ c& d \end{pmatrix}\in \SL_2(F) \subset M$, we can check that
$$g^{-1}[r_1,r_2,r_3,0,0] g=[r_1', r_2', r_3', r_4',r_5'],$$
where $r_1'=a r_1- cr_2, r_2'=-b r_1 + dr_2, r_3'-r_1'r_2'=r_3-r_1r_2$. 

Consider the map $\overline \pr: J(\BA)=\SL_2(\BA)\ltimes V(\BA)\ra \SL_2(\BA)\ltimes \sH(\BA)$, $$(g,v)\mapsto (g^*, \pr(v)), g\in \SL_2(\BA), v\in V(\BA).$$
where $g^*=\begin{pmatrix} a& -b \\ -c & d \end{pmatrix}=d_1 g d_1^{-1}$, where $d_1=\diag(1,-1)\in \GL_2(F)$. From the above discussion, the map $\ov\pr$ is a group homomorphism and its kernel is also $ V_1(\BA)$. We will also view $\ov{\textrm{pr}}$ as a homormophism $\wt{\SL}_2(\BA)\ltimes V(\BA)\ra \wt{\SL}_2(\BA)\ltimes\sH(\BA)$. 

 In the following, we will also view $\omega_\psi$ as a representation of $\wt{\SL_2}(\BA)\ltimes V(\BA)$ via the projection map $\ov{\pr}$. For $\phi\in \CS(\BA)$, we form the theta series 
$$\wt \theta_\phi(vg)=\sum_{\xi\in F}\omega_\psi(vh)\phi(\xi),v\in V(\BA),g\in \wt\SL_2(\BA).$$

Note that given a genuine cusp form $\wt \varphi$ on $\wt \SL_2(\BA)$, the product 
$$\wt \varphi(g)\wt \theta_\phi(vg),v\in V(\BA),g\in \wt\SL_2(\BA)$$
can be viewed as a function on $J(\BA)=\SL_2(\BA)\ltimes V(\BA)$.

\subsection{An Eisenstein series on \texorpdfstring{$\RG_2$}{Lg}}
Let $\tau$ be a cuspidal automorphic representation on $\GL_2(\BA)$. We will view $\tau$ as a representation of $M'(\BA)$ via the identification $M'\cong \GL_2$. We then consider the induced representation $I(s,\tau)=\Ind_{P'(\BA)}^{\bG_2(\BA)}(\tau\otimes \delta_{P'}^s)$. A section $f_s\in I(s,\tau)$ is a smooth function satisfying
$$f_s(v'm'g)=\delta_{P'}(m')^sf_s(g), \forall v'\in V'(\BA), m'\in M'(\BA),g\in \RG_2(\BA).$$

 For $f_s\in I(s,\tau)$, we consider the Eisenstein series
$$E(g,f_s)=\sum_{\delta\in P'(F)\setminus \bG_2(F)}f_s(\delta g),g\in \bG_2(\BA).$$

\section{A global integral}
Let $\wt\pi$ be a genuine cuspidal automorphic representation on $\wt\SL_2(\BA),$ and $\tau$ be a cuspidal  automorphic representation of $\GL_2(\BA)$. For $\wt \varphi\in V_\pi, \phi\in \CS(\BA)$ and $f_s\in I(s,\tau)$, we consider the integral
$$I(\wt \varphi,\phi,f_s)=\int_{\SL_2(F)\setminus \SL_2(\BA)}\int_{V(F)\setminus V(\BA)}\wt \varphi(g) \wt \theta_\phi(vg) E(vg,f_s)dvdg.$$
Let $\gamma=w_\beta w_\alpha w_\beta w_\alpha\in \bG_2(F)$.
\begin{thm}\label{thm: eulerian}
The integral $I(\wt \varphi,\phi,f_s)$ is absolutely convergent when $\Re(s)\gg 0$ and can be meromorphically continued to all $s\in \BC$. Moreover, when $\Re(s)\gg0$, we have 
$$I(\tilde \varphi,\phi,f_s)=\int_{N_{\SL_2}(\BA)\setminus \SL_2(\BA)}\int_{U_{\alpha+\beta}(\BA)\setminus V(\BA)} W_{\wt \varphi}(g)\omega_\psi(vg)\phi(1)W_{f_s}(\gamma vg)dvdg, $$
where $$ W_{\wt \varphi}(g)=\int_{F\setminus \BA}\wt \varphi(\bx_\beta(r)g)\psi(r)dr,$$ 
and 
$$ W_{f_s}(\gamma vg)=\int_{F\setminus \BA}f_s(\bx_\alpha(r)\gamma vg)\psi(-2r)dr.$$
\end{thm}
\begin{proof}
The first assertion is standard. We only show that the above integral is Eulerian when $\Re(s)\gg0$. Unfolding the Eisenstein series, we can get 
$$I(\wt \varphi,\phi,f_s)=\sum_{\delta\in P'(F)\setminus \bG_2(F)/P(F)}\int_{\SL_2^\delta(F)\setminus \SL_2(\BA)}\int_{V^\delta(F)\setminus V(\BA)}\wt \varphi(g)\wt \theta_\phi(vg)f_s(\delta vg)dvdg,$$
where  $X^\delta=\delta^{-1} P'\delta\cap X$ for $X\subset G_2(F)$. We can check that a set of representatives of the double coset $P'(F)\setminus \bG_2(F)/P(F)$ can be taken as $\wpair{1,w_\beta w_\alpha, \gamma= w_\beta w_\alpha w_\beta w_\alpha}.$ For $\delta=1, w_\beta w_\alpha,$ or $ \gamma=w_\beta w_\alpha w_\beta w_\alpha$, denote 
$$I_\delta=\int_{\SL_2^\delta(F)\setminus \SL_2(\BA)}\int_{V^\delta(F)\setminus V(\BA)}\wt \varphi(g)\wt \theta_\phi(vg)f_s(\delta vg)dvdg.$$

If $\delta=1$, the above integral $I_\delta$ has an inner integral 
$$\int_{U_{2\alpha+\beta}(F)\setminus U_{2\alpha+\beta}(\BA)}\wt \theta_\phi(\bx_{2\alpha+\beta}(r)vg)f_s(\bx_{2\alpha+\beta}(r)vg)dr,$$
which is zero because $f_s(\bx_{2\alpha+\beta}(r)vg)=f_s(vg) , \wt \theta_\phi(\bx_{2\alpha+\beta}(r)vg)=\psi(r)\tilde \theta_\phi(vg)$ and $\int_{F\setminus \BA}\psi(r)dr=0.$ The last equation follows from the fact that $\psi$ is non-trivial.

We next consider the term when $\delta=w_\beta w_\alpha$. We write $$\wt \theta_\phi(vg)=\omega_\psi(vg)\phi(0)+\sum_{\xi\in F^\times}\omega_\psi(vg)\phi(\xi).$$ The contribution of the first term to the integral $I_{\delta}$ is 
$$\int_{\SL_2^\delta(F)\setminus \SL_2(\BA)}\int_{V^\delta(F)\setminus V(\BA)}\wt \varphi(g)\omega_\psi(vg)\phi(0)f_s(\delta vg)dvdg. $$
Note that $\delta \bx_\beta(r)\delta^{-1}\subset U_{2\alpha+\beta}\subset V'$, we have $f_s(\delta v \bx_\beta(r)g)=f_s(\delta \bx_\beta(-r)v\bx_\beta(r)g)$. On the other hand, we have $\omega_\psi(\bx_\beta(r)vg)\phi(0)=\omega_\psi(vg)\phi(0)$ . After a changing variable on $v$, we can see that the above integral contains an inner integral 
$$\int_{F\setminus \BA}\wt \varphi(\bx_\beta(r)vg)dr,$$
which is zero since $\wt \varphi$ is cuspidal. Thus the contribution of the term $\omega_\psi(vg)\phi(0)$ is zero when $\delta=w_\beta w_\alpha$. The contribution of $\sum_{\xi\in F^\times}\omega_\psi(vg)\phi(\xi) $ is 
$$\int_{\SL_2^\delta(F)\setminus \SL_2(\BA)}\int_{V^\delta(F)\setminus V(\BA)}\wt \varphi(g) \sum_{\xi\in F^\times}\omega_\psi(vg)\phi(\xi)f_s(\delta vg)dvdg. $$
We consider the inner integral on $U_{\alpha+\beta}(F)\setminus U_{\alpha+\beta}(\BA)$. Note that $U_{\alpha+\beta}\subset V$ and $\delta U_{\alpha+\beta} \delta^{-1}=U_{2\alpha +\beta}\subset V'$, we get $f_s(\delta \bx_{\alpha+\beta}(r)vg)=f_s(\delta vg)$. On the other hand, we have 
$\omega_{\psi}(\bx_{\alpha+\beta}(r)vg)\phi(\xi)=\psi(-2r\xi)\omega_\psi(vg)\phi(\xi)$. Thus the above integral has an inner integral
$$\int_{F\setminus \BA}\sum_{\xi\in F^\times}\psi(-2r\xi)\omega_\psi(vg)\phi(\xi)dr=\sum_{\xi\in F^\times}\omega_\psi(vg)\phi(\xi)\int_{F\setminus \BA}\psi(-2r\xi)dr=0.$$
Thus when $\delta=w_\beta w_\alpha$, the corresponding term is zero.  Thus we get
$$I(\wt \varphi,\phi,f_s)=\int_{\SL_2^\gamma(F)\setminus \SL_2(\BA)}\int_{V^\gamma(F)\setminus V(\BA)}\wt \varphi(g)\wt \theta_\phi(vg)f_s(\gamma vg)dvdg.$$  We have $\SL_2^\gamma=B_{\SL_2}$ and $V^\gamma=U_{\alpha+\beta}$. We decompose $\wt \theta_\phi$ as 
 $$\wt \theta_\phi(vg)=\omega_\psi(vg)\phi(0)+\sum_{\xi\in F^\times}\omega_\psi(vg)\phi(\xi)=\omega_\psi(vg)\phi(0)+\sum_{a\in F^\times}\omega_\psi(t(a)vg)\phi(1).$$
Recall that $t(a)=\diag(a,a^{-1})$. Since $\gamma U_\beta \gamma^{-1}\subset U_{3\alpha+\beta}\subset V'$, we have 
$$f_s(\gamma v\bx_\beta(r)g)=f_s(\gamma \bx_\beta(-r)v\bx_\beta(r)g). $$
On the other hand we have $\omega_\psi(v\bx_\beta(r)g)\phi(0)=\omega_\psi(\bx_\beta(-r)v\bx_\beta(r)g)\phi(0) $. Thus after a changing variable on $v$, we can get that the contribution of $\omega_\psi(vg)\phi(0)$ to $I(\wt \varphi,\phi,f_s)$ has an inner integral 
$$\int_{F\setminus \BA}\wt \varphi(\bx_\beta(r)g)dr,$$
which is zero by the cuspidality of $\wt \varphi$. Thus we get 
$$I(\wt \varphi,\phi,f_s)=\int_{B_{\SL_2}(F)\setminus \SL_2(\BA)}\int_{U_{\alpha+\beta}(F)\setminus V(\BA)}\wt \varphi(g)\sum_{a\in F^\times}\omega_\psi(t(a)vg)\phi(1)f_s(\gamma vg)dvdg.$$ 
Collapsing the summation with the integration, we then get 
\begin{align*}&\quad I(\wt \varphi,\phi,f_s)\\
&=\int_{N_{\SL_2}(F)\setminus \SL_2(\BA)}\int_{U_{\alpha+\beta}(F)\setminus V(\BA)}\wt \varphi(g)\omega_\psi(vg)\phi(1)f_s(\gamma vg)dvdg\\
&=\int_{N_{\SL_2}(\BA)\setminus \SL_2(\BA)}\int_{U_{\alpha+\beta}(F)\setminus V(\BA)}\int_{F\setminus \BA}\wt \varphi(\bx_\beta(r)g)\omega_\psi(v\bx_\beta(r)g)\phi(1)f_s(\gamma v\bx_\beta(r)g)drdvdg.
\end{align*}
Note that we have $\omega_\psi(v\bx_\beta(r)g)\phi(1)=\omega_\psi(\bx_\beta(r)\bx_\beta(-r)v\bx_\beta(r)g)\phi(1)=\psi(r) \omega_\psi(\bx_\beta(-r)v\bx_\beta(r)g)\phi(1)$. On the other hand, we have $\gamma \bx_\beta(r)\gamma^{-1}\subset U_{3\alpha+\beta}\subset V'$. Thus $f_s(\gamma v\bx_\beta(r)g)=f_s(\gamma \bx_\beta(-r)v\bx_\beta(r)g) .$ After a changing of viarable on $v$, we get 
$$I(\wt \varphi,\phi,f_s)=\int_{N_{\SL_2}(\BA)\setminus \SL_2(\BA)}\int_{U_{\alpha+\beta}(F)\setminus V(\BA)}W_{\wt \varphi}(g)\omega_\psi(vg)\phi(1)f_s(\gamma vg)dvdg,$$
where 
$$ W_{\wt \varphi}(g)=\int_{F\setminus \BA}\tilde \varphi(\bx_\beta(r)g)\psi(r)dr.$$ 
We can further decompose the above integral as 
\begin{align*}
&I(\wt \varphi,\phi,f_s)\\
=&\int_{N_{\SL_2}(\BA)\setminus \SL_2(\BA)}\int_{U_{\alpha+\beta}(\BA)\setminus V(\BA)}\int_{F\setminus \BA}W_{\wt \varphi}(g)\omega_\psi(\bx_{\alpha+\beta}(r)vg)\phi(1)f_s(\gamma \bx_{\alpha+\beta}(r)vg)dr dvdg.
\end{align*}
Note that $\omega_\psi(\bx_{\alpha+\beta}(r)vg)\phi(1)=\psi(-2r)\omega_\psi(vg)\phi(1) $ and $ f_s(\gamma \bx_{\alpha+\beta}(r)vg)=f_s(\bx_\alpha(r)\gamma vg)$ since $\gamma \bx_{\alpha+\beta}(r)\gamma^{-1}=\bx_\alpha(r)$. We then get 
$$I(\wt \varphi,\phi,f_s)=\int_{N_{\SL_2}(\BA)\setminus \SL_2(\BA)}\int_{U_{\alpha+\beta}(\BA)\setminus V(\BA)} W_{\wt \varphi}(g)\omega_\psi(vg)\phi(1)W_{f_s}(\gamma vg)dvdg, $$
where $$ W_{f_s}(\gamma vg)=\int_{F\setminus \BA}f_s(\bx_\alpha(r)\gamma vg)\psi(-2r)dr.$$
This concludes the proof.
\end{proof}

\section{Unramified calculation}\label{sec:unramified}
In this section, let $F$ be a $p$-adic field with $p\ne 2$. Let $\fo$ be the ring of integers of $F$, and let $p$ be a uniformizer of $\fo$ by abuse of notation. Let $q$ be the cardinality of the residue field $\fo/(p)$.

\subsection{Local Weil representations}
Let $\psi$ be an additive character of $F$ and let $\gamma(\psi)$ be the Weil index and let $\mu_\psi(a)=\frac{\gamma(\psi)}{\gamma(\psi_a)}$. Let $\omega_\psi$ be the Weil representation of $\wt\SL_2(F)\ltimes V$ on $\CS(F)$ via the projection $\wt \SL_2(F)\ltimes V\ra \wt \SL_2(F)\ltimes \sH$.
For $\phi\in \CS(F)$, we have the well-know formulas:
\begin{align*}
(\omega_{\psi}(w^1)\phi)(x)&=\gamma(\psi)\hat \phi(x),\\
(\omega_\psi(n(b))\phi)(x)&=\psi(bx^2)\phi(x), b\in F\\
(\omega_\psi(t(a))\phi)(x)&=|a|^{1/2} \mu_\psi(a)\phi(ax), a\in F^\times\\
(\omega_{\psi}((r_1,r_2,r_3))\phi)(x)&=\psi\left( r_3-2xr_2-r_1r_2\right) \phi(x+r_1), (r_1,r_2,r_3)\in \sH(F).
\end{align*}
where  $\hat \phi(x)=\int_F \phi(y)\psi(2xy)dy$ is the Fourier transform of $\phi$ with respect to $\psi$. Note that under the embedding $\SL_2(F)\incl \bG_2(F)$, we have $w^1=w_\beta,n(b)=\bx_\beta(b)$ and $t(a)=h(a,a^{-1})$.

\subsection{Unramified calculation}
In this subsection, we compute the local integral in last section. The strategy is similar as the unramified calculation in \cite{Gi95}. 

Let $\wt \pi$ be an unramified genuine representation of $\wt \SL_2(F)$ with Satake parameter $a$, and let $\tau$ be an unramified irreducible representation of $\GL_2(F)$ with Satake parameters $b_1,b_2$. Let $\wt W\in \CW(\tilde \pi, \psi)$ with $\wt W(1)=1$. Let $v_0\in V_\tau$ be an unramified vector and $\lambda\in \Hom_{N}(V_\tau,\psi)$ such that $\lambda(v_0)=1$. Let $f_s:G_2\ra V_\tau$ be the unramified section in $I(s,\tau)$ with $f_s(e)=v_0$. Let $$W_{f_s}:G_2\times \GL_2(F)\ra \BC$$
be the function $W_{f_s}(g,a)=\lambda(\tau(a)f_s(g)).$ We will write $W_{f_s}(g)$ for $W_{f_s}(g,1)$ in the following. By assumption and Shintani formula, we have 
\begin{align}
\begin{split}\label{eq: shintani}
W_{f_s}(h(p^k,p^l))&=q^{-3s(2k+l)}\lambda(\tau(\diag(p^{k+l},p^k))v_0)\\
&=q^{-3s(2k+l)}W_{v_0}(\diag(p^{k+l},p^k))\\
&=\left\{ \begin{array}{lll}q^{-3s(2k+l)}\frac{(b_1b_2)^kq^{-l/2}}{b_1-b_2}(b_1^{l+1}-b_2^{l+1}), & \textrm{ if } l\ge 0,\\ 0, & \textrm{ if } l<0. \end{array} \right.
\end{split}
\end{align}

Let $\phi\in \CS(F)$ be the characteristic function of $\fo$. We need to compute the integral 
$$I(\wt W, W_{f_s},\phi)=\int_{N_2\setminus \SL_2(F)}\int_{U_{\alpha+\beta}\setminus V}\wt W(g)\omega_\psi(vg)\phi(1)W_{f_s}(\gamma vg)dvdg.$$
In the following, we fix the Haar measure such that $\vol(dr,\fo)=1$. Thus $\vol(d^*r,\fo^\times)=1-q^{-1}$.

Using the Iwasawa decomposition $\SL_2(F)=N_2(F)A_2(F)\SL_2(\fo)$, we have 
\begin{align*}&I(\wt W, W_{f_s},\phi)\\
=&\int_{F^\times}\int_{F^4}\wt W(t(a))\omega_\psi([r_1,0,r_3]t(a))\phi(1)W_{f_s}(\gamma (r_1,0,r_3,r_4,r_5)t(a))|a|^{-2}dr_1dr_3dr_4dr_5 d^\times a\\
=&\int_{F^\times}\int_{F^4}\wt W(t(a))\omega_\psi(t(a)[r_1,0,r_3])\phi(1)W_{f_s}(\gamma t(a) (r_1,0,r_3,r_4,r_5))|a|^{-3}dr_1dr_3dr_4dr_5 d^\times a
\end{align*}
If $\wt W(t(a))\ne 0$, then $|a|\le 1$. On the other hand, we have 
$$\omega_\psi(t(a)[r_1,0,r_3])\phi(1)=\mu_\psi(a)|a|^{1/2}\psi(r_3)\phi(a+r_1).$$
If $\phi(a+r_1)\ne 0$ and $a\in \fo$, then $r_1\in \fo$. Thus the domain for $a$ and $r_1$ in the above integral is $\wpair{a\in F^\times\cap \fo,r_1\in \fo}$. Note that $\gamma t(a)=h(1,a)\gamma=h(1,a)w_\beta w_\alpha w_\beta w_\alpha$. Thus, if we conjugate $w_\alpha \bx_\alpha(r_1)$ to the right side, we can get 
$$ h(1,a)\gamma [r_1,0,r_3,r_4,r_5]=h(1,a)w_\beta w_\alpha w_\beta \bx_{\alpha+\beta}(-r_3)\bx_\beta(-r_4-3r_1r_3)\bx_{3\alpha+2\beta}(r_5)w_\alpha \bx_\alpha(r_1).$$
Since $w_\alpha \bx_\alpha(r_1)\in K$ for $r_1\in \fo$, by changing of variables, we get 
\begin{align*}
   &I(\wt W, W_{f_s},\phi)\\
=&\int_{|a|\le 1}\wt W(t(a))|a|^{-5/2}\mu_\psi(a) \\
&\quad \cdot \int_{F^3} W_{f_s}(h(1,a)w_\beta w_\alpha w_\beta \bx_{\alpha+\beta}(r_3)\bx_\beta(r_4)\bx_{3\alpha+2\beta}(r_5))\psi(-r_3)dr_3dr_4dr_5 d^*a\\
=&\sum_{n\ge 0} \wt W(t(p^n)) q^{5n/2}\mu_\psi(p^n) J(n),
\end{align*}
where 
$$J(n)= \int_{F^3} W_{f_s}(h(1,p^n)w_\beta w_\alpha w_\beta \bx_{\alpha+\beta}(r_3)\bx_\beta(r_4)\bx_{3\alpha+2\beta}(r_5))\psi(-r_3)dr_3dr_4dr_5.$$

By dividing the domain of $r_3$ into two parts, we can write $J(n)=J_1(n)+J_2(n)$, where 
\begin{align*}J_1(n)&=\int_{|r_3|\le 1}\int_{F^2} W_{f_s}(h(1,p^n)w_\beta w_\alpha w_\beta \bx_{\alpha+\beta}(r_3)\bx_\beta(r_4)\bx_{3\alpha+2\beta}(r_5))\psi(-r_3)dr_3dr_4dr_5\\
&=\int_{F^2} W_{f_s}(h(1,p^n)w_\beta w_\alpha w_\beta \bx_\beta(r_4)\bx_{3\alpha+2\beta}(r_5))dr_4dr_5,
\end{align*}
and 
$$J_2(n)=\int_{|r_3|> 1}\int_{F^2} W_{f_s}(h(1,p^n)w_\beta w_\alpha w_\beta \bx_{\alpha+\beta}(r_3)\bx_\beta(r_4)\bx_{3\alpha+2\beta}(r_5))\psi(-r_3)dr_3dr_4dr_5.$$

\begin{lem}
Set 
$$I(n)=\int_F W_{f_s}(h(1,p^n)w_\beta \bx_\beta(r))dr.$$
Then
\begin{align*}I(n)&=\frac{q^{-(3s+1/2)n}}{b_1-b_2}\left[(b_1^{n+1}-b_2^{n+1}) \right.\\
&\left. +(1-q^{-1})\frac{b_1b_2X}{(1-b_1X)(1-b_2X)}(b_1^n-b_2^n-b_1^{n+1}X+b_2^{n+1}X+b_1X(b_1b_2X)^n-b_2X(b_1b_2X)^n)\right],
\end{align*}
 where $X=q^{-(3s-3/2)}$.
\end{lem}
\begin{proof}
We have 
\begin{align*}
I(n)&=\int_F W_{f_s}(h(1,p^n)w_\beta \bx_\beta(r))dr\\
&=\int_{|r|\le 1}W_{f_s}(h(1,p^n)w_\beta \bx_\beta(r))dr\\
&+\int_{|r|> 1}W_{f_s}(h(1,p^n)w_\beta \bx_\beta(r))dr\\
&=W_{f_s}(h(1,p^n))+\int_{|r|> 1}W_{f_s}(h(1,p^n)w_\beta \bx_\beta(r))dr.
\end{align*}
To deal with the integral when $|r|>1$, we consider the following Iwasawa decomposition of $w_\beta \bx_\beta(r)$:
$$w_\beta \bx_\beta(r)=\bx_\beta(-r^{-1})h(-r^{-1},-r)\bx_{-\beta}(r^{-1}).$$
Since $\bx_{-\beta}(r^{-1})$ is in the maximal compact subgroup for $|r|>1$, we have 
$$W_{f_s}(h(1,p^n)w_\beta \bx_\beta(r))=W_{f_s}(h(1,p^n)\bx_{\beta}(-r^{-1})h(-r^{-1},-r))=W_{f_s}(h(1,p^n)h(r^{-1},r)),$$
where we used $U_\beta\subset V'$. For $|r|>1$, we can write $r= p^{-m}u$ for some $m\ge 1$ and $u\in \fo^\times$. We then have $dr=q^{m}du$. Note that $\vol(\fo^\times)=1-q^{-1}$. Thus we have 
$$I(n)=W_{f_s}(h(1,p^n))+\sum_{m\ge 1}(1-q^{-1})q^mW_{f_s}(h(p^m,p^{n-m})).$$
Note that $h(p^m,1)\mapsto \diag(p^m,p^m)$ under the isomorphism $M'\cong \GL_2$. Thus we have 
$$ W_{f_s}(h(p^m,1)h(1,p^{n-m}))=q^{-6sm}\omega_\tau(p)^mW_{f_s}(h(1,p^{n-m})).$$
Thus we get 
$$I(n)=W_{f_s}(h(1,p^n))+\sum_{m\ge 1}(1-q^{-1})q^{(-6s+1)m}\omega_\tau(p)^mW_{f_s}(h(1,p^{n-m})).$$
By \eqref{eq: shintani}, we have
\begin{align*}
W_{f_s}(h(1,p^{n-m}))&=\left\{ \begin{array}{lll} \frac{q^{-3s(n-m)-(n-m)/2}}{b_1-b_2}(b_1^{n-m+1}-b_2^{n-m+1}), & \textrm{ if } n\ge m,\\ 0, & \textrm{ if } n<m. \end{array} \right.
\end{align*}
 Thus for $n\ge 1$, we have
\begin{align*}
I(n)&=\frac{q^{-(3s+1/2)n}}{b_1-b_2}\left((b_1^{n+1}-b_2^{n+1})+\sum_{m= 1}^n(1-q^{-1})q^{-(3s-3/2)m}(b_1^{n+1}b_2^m-b_2^{n+1}b_1^m)\right).
\end{align*}
Thus result can be computed using the geometric summation formula. One can check that the given formula also satisfies $I(0)=1$.
\end{proof}

\begin{lem}
We have $$J_1(n)=\frac{1-q^{-6s+1}b_1b_2}{1-q^{-6s+2}b_1b_2}I(n).$$
\end{lem}
\begin{proof}
To compute $J_1(n)$, we break up the domain of integration in $r_4$ and get
\begin{align*}
J_1(n)&=\int_{F}\int_{|r_4|\le 1} W_{f_s}(h(1,p^n)w_\beta w_\alpha w_\beta \bx_\beta(r_4)\bx_{3\alpha+2\beta}(r_5))dr_4dr_5\\
&+\int_{F}\int_{|r_4|> 1} W_{f_s}(h(1,p^n)w_\beta w_\alpha w_\beta \bx_\beta(r_4)\bx_{3\alpha+2\beta}(r_5))dr_4dr_5\\
&:=J_{11}(n)+J_{12}(n),
\end{align*}
where
\begin{align*}J_{11}(n)=&\int_{F}\int_{|r_4|\le 1}W_{f_s}(h(1,p^n)w_\beta w_\alpha w_\beta \bx_\beta(r_4)\bx_{3\alpha+2\beta}(r_5))dr_4 dr_5 \\
=&\int_{F}\int_{|r_4|\le 1}W_{f_s}(h(1,p^n) w_\beta w_\alpha w_\beta \bx_{3\alpha+2\beta}(r_5)w_\beta^{-1}w_\alpha^{-1} w_\alpha w_\beta \bx_\beta(r_4))dr_4dr_5\\
=&\int_{F}W_{f_s}(h(1,p^n)w_\beta \bx_\beta (r_5))dr_5\\
=&I(n),
\end{align*}
and 

\begin{align*}J_{12}(n)=&\int_{F}\int_{|r_4|> 1}W_{f_s}(h(1,p^n)w_\beta w_\alpha w_\beta \bx_\beta(r_4)\bx_{3\alpha+2\beta}(r_5))dr_4 dr_5 \\
=&\int_{F}\int_{|r_4|> 1}W_{f_s}(h(1,p^n) w_\beta w_\alpha w_\beta \bx_{3\alpha+2\beta}(r_5)w_\beta^{-1}w_\alpha^{-1} w_\alpha w_\beta \bx_\beta(r_4))dr_4dr_5\\
=&\int_{F}\int_{|r_4|> 1}W_{f_s}(h(1,p^n) w_\beta  \bx_{\beta}(r_5) w_\alpha w_\beta \bx_\beta(r_4))dr_4dr_5.
\end{align*}
We have the Iwasawa decomposition of $w_\beta \bx_\beta(r_4)$:
$$w_\beta \bx_\beta(r_4)=\bx_\beta(-r_4^{-1})h(-r_4^{-1},-r_4)\bx_{-\beta}(r_4^{-1}).$$
Since $\bx_{-\beta}(r_4^{-1})$ is in the maximal compact subgroup for $|r_4|>1$, we then get 
\begin{align*}
J_{12}(n)&=\int_{F}\int_{|r_4|> 1}W_{f_s}(h(1,p^n) w_\beta  \bx_{\beta}(r_5) w_\alpha \bx_\beta(-r_4^{-1})h(r_4^{-1},r_4))dr_4dr_5\\
&=\int_F \int_{|r_4|>1}W_{f_s}(h(1,p^n) h(r_4^{-1},1)w_\beta  \bx_{\beta}(r_4^{-1}r_5))dr_4dr_5\\
&=\int_F \int_{|r_4|>1}|r_4|W_{f_s}(h(1,p^n) h(r_4^{-1},1)w_\beta  \bx_{\beta}(r_5))dr_4dr_5\\
&=\sum_{m\ge 1}(1-q^{-1})q^{2m} \int_F W_{f_s}(h(p^m,1)h(1,p^n)w_\beta \bx_\beta(r_5))dr_5,
\end{align*}
where in the second equality, we conjugated $\bx_\beta(-r_4^{-1})h(r_4^{-1},r_4)$ to the left, and in the third equality, we wrote $r_4=p^{-m}u$ for $m\ge 1, u\in \fo^\times$ and used $dr_4=q^mdu, \vol(\fo^\times)=1-q^{-1}$.  Note that $h(p^m,1)$ is in the center of $M'$, and thus  
$$W_{f_s}(h(p^m,1)g)=q^{-6sm}\omega_\tau(p)^mW_{f_s}(g),$$
we get 
$$J_{12}(n)=(1-q^{-1})\sum_{m\ge 1}q^{-6sm+2m}\omega_\tau(p)^m\int_F W_{f_s}(h(1,p^n)w_\beta\bx_\beta(r_5))dr_5.$$
Thus we get 
$$J_1(n)=I(n)+\sum_{m\ge 1}(1-q^{-1})q^{(-6s+2)m}(b_1b_2)^m I(n). $$
A simple calculation gives the formula of $J_1(n)$.
\end{proof}
We next consider the term 
$$J_2(n)=\int_{|r_3|> 1}\int_{F^2} W_{f_s}(h(1,p^n)w_\beta w_\alpha w_\beta \bx_{\alpha+\beta}(r_3)\bx_\beta(r_4)\bx_{3\alpha+2\beta}(r_5))\psi(-r_3)dr_3dr_4dr_5.$$
For $|r_3|>1$, we can write $r_3\in p^{-m}u$ with $m\ge 1, u\in \fo^\times$.
We then have, 
$$J_2(n)=\int_{F^2}\sum_{m\ge 1}q^m W_{f_s}(h(1,p^n)w_\beta w_\alpha w_\beta \bx_{\alpha+\beta}(p^{-m}u)\bx_\beta(r_4)\bx_{3\alpha+2\beta}(r_5))\psi(-p^{-m}u)dudr_4dr_5.$$
Write $\bx_{\alpha+\beta}(p^{-m}u)=h(u,u^{-1})\bx_{\alpha+\beta}(p^{-m})h(u^{-1},u)$, and by conjugation and changing of variables, we get 
$$J_2(n)=\int_{F^2}\sum_{m\ge 1}q^m W_{f_s}(h(u^{-1},p^n)w_\beta w_\alpha w_\beta \bx_{\alpha+\beta}(p^{-m})\bx_\beta(r_4)\bx_{3\alpha+2\beta}(r_5))\psi(-p^{-m}u)dudr_4dr_5,$$
where we used $h(u,u^{-1})$ is in the maximal compact subgroup of $\RG_2(F)$. Since $h(u^{-1},1)$ maps to the center of $M'$ and $ |\omega_\tau(u)|=1 $, we have 
$$W_{f_s}(h(u^{-1},p^n)w_\beta w_\alpha w_\beta \bx_{\alpha+\beta}(p^{-m})\bx_\beta(r_4)\bx_{3\alpha+2\beta}(r_5))=W_{f_s}(1,p^n)w_\beta w_\alpha w_\beta \bx_{\alpha+\beta}(p^{-m})\bx_\beta(r_4)\bx_{3\alpha+2\beta}(r_5)) .$$
Thus we get 
$$J_2(n)=\int_{F^2}\sum_{m\ge 1}q^m W_{f_s}(h(1,p^n)w_\beta w_\alpha w_\beta \bx_{\alpha+\beta}(p^{-m})\bx_\beta(r_4)\bx_{3\alpha+2\beta}(r_5))\psi(-p^{-m}u)dudr_4dr_5.$$
Since
$$\int_{\fo^\times}\psi(p^ku)du=\left\{\begin{array}{lll}1-q^{-1}, & \textrm{ if } k\ge 0,\\ -q^{-1}, & \textrm{ if } k=-1, \\ 0, & \textrm{ if } k\le -2, \end{array}\right.$$
 we get $J_2(n)=-R(n)$, where
\begin{align*}
R(n)=\int_{F^2} W_{f_s}(h(1,p^n)w_\beta w_\alpha w_\beta \bx_{\alpha+\beta}(p^{-1})\bx_\beta(r_4)\bx_{3\alpha+2\beta}(r_5))dr_4dr_5.
\end{align*}
To evaluate $R(n)$, we split the domain of $r_4$, and write $R(n)=R_1(n)+R_{2}(n)$, where 
\begin{align*}R_1(n)&=\int_{|r_4|\le 1} \int_{F} W_{f_s}(h(1,p^n)w_\beta w_\alpha w_\beta \bx_{\alpha+\beta}(p^{-1})\bx_\beta(r_4)\bx_{3\alpha+2\beta}(r_5))dr_4dr_5,\\
&= \int_{F} W_{f_s}(h(1,p^n)w_\beta w_\alpha w_\beta \bx_{\alpha+\beta}(p^{-1})\bx_{3\alpha+2\beta}(r_5))dr_5,
\end{align*}
and 
$$R_2(n)=\int_{|r_4|>1} \int_{F} W_{f_s}(h(1,p^n)w_\beta w_\alpha w_\beta \bx_{\alpha+\beta}(p^{-1})\bx_\beta(r_4)\bx_{3\alpha+2\beta}(r_5))dr_4dr_5 .$$
We now compute $R_1(n)$. We conjugate $w_\alpha w_\beta \bx_{\alpha+\beta}(p^{-1})$ to the right and then get
\begin{align*}R_1(n)&=\int_{F} W_{f_s}(h(1,p^n)w_\beta \bx_{\beta}(r_5)w_\alpha w_\beta\bx_{\alpha+\beta}(p^{-1}))dr_5\\
&=\int_{F} W_{f_s}(h(1,p^n)w_\beta \bx_{\beta}(r_5)w_\alpha \bx_{\alpha}(-p^{-1}))dr_5
\end{align*}
Next, we use the Iwasawa decomposition of $w_\alpha \bx_\alpha(p^{-1})$:
$$w_\alpha \bx_\alpha (-p^{-1})=\bx_\alpha(p)h(p^{-1},p^2)\bx_{-\alpha}(-p)$$
to get 
$$R_1(n)=\int_F W_{f_s}(h(1,p^n)w_\beta \bx_\beta(r_5)\bx_\alpha(p)h(p^{-1},p^2))dr_5.$$
Next, we use the commutator relation
$$\bx_\beta(r_5)\bx_{\alpha}(p)=\bx_{\alpha+\beta}(pr_5)u\bx_\alpha(p)\bx_\beta(r_5),$$
where $u$ is in the root space of $2\alpha+\beta,3\alpha+\beta,3\alpha+2\beta$. Then we get 
$$R_1(n)=\int_F W_{f_s}(h(1,p^n)w_\beta \bx_{\alpha+\beta}(pr_5)u\bx_\alpha(p)\bx_\beta(r_5)h(p^{-1},p^2))dr_5.$$
Note that $w_\beta u \bx_\alpha(r) w_\beta(1)\in V'$, and $h(1,p^n)w_\beta \bx_{\alpha+\beta}(pr_5)(h(1,p^n)w_\beta)^{-1}=\bx_\alpha(-p^{n+1}r_5),$ and $W_{f_s}(\bx_\alpha(r)g)=\psi(2r)W_{f_s}(g)$, we get
\begin{align*}R_1(n)&=\int_F W_{f_s}(h(1,p^n)w_\beta \bx_\beta(r_5)h(p^{-1},p^2))\psi(-2p^{n+1}r_5)dr_5\\
&=\int_F W_{f_s}(h(p^2,1)h(1,p^{n-1})w_\beta \bx_\beta(p^3r_5))\psi(-2p^{n+1}r_5)dr_5\\
&=q^{-12s+3}\omega_\tau(p^2)\int_{F}W_{f_s}(h(1,p^{n-1})w_\beta \bx_\beta(r_5))\psi(-2 p^{n-2}r_5)dr_5,
\end{align*}
where the last equality comes from a changing of variable on $r_5$ and the fact that $h(p^2,1)\mapsto \diag(p^2,p^2)$ under the isomorphism $M'\cong \GL_2$. We next break up the integral on $r_5$ and get
\begin{align*}
R_1(n)&=q^{-12s+3}\omega_\tau(p^2)W_{f_s}(h(1,p^{n-1}))\int_{|r_5|\le 1}\psi(-2p^{n-2}r_5)dr_5\\
&+q^{-12s+3}\omega_\tau(p^2)\int_{|r_5|>1}W_{f_s}(h(1,p^{n-1})w_\beta \bx_\beta(r_5))\psi(-2p^{n-2}r_5)dr_5.
\end{align*}
Using the Iwasawa decomposition of $w_\beta \bx_\beta(r_5)$, we have
\begin{align*}R_1(n)=q^{-12s+3}\omega_\tau(p^2)&\left(W_{f_s}(h(1,p^{n-1}))\int_{|r_5|\le 1}\psi(-2p^{n-2}r_5)dr_5 \right.\\
&\left.+\sum_{m=1}^\infty W_{f_s}(h(p^{m},p^{n-m-1}))q^m\int_{\fo^\times}\psi(-2p^{n-m-2}u)du \right).
\end{align*}
\begin{lem}
We have $R_1(n)=0$ if $n\le 1$, and 
$$R_1(n)=q^{-12s+3}\omega_\tau(p)^2I(n-1)-q^{-6s(n+1)+n+2}\omega_\tau(p)^{n+1},
$$
for $n\ge 2$.
\end{lem}
\begin{proof}
Note that $\int_{|r|\le 1}\psi(p^kr)dr=0$ if $k<0$ and $\int_{|r|\le 1}\psi(p^kr)dr=1 $ if $k\ge 0$. Moreover, we have 
$$\int_{\fo^\times}\psi(p^ku)du=\left\{\begin{array}{lll}1-q^{-1}, & \textrm{ if } k\ge 0,\\ -q^{-1}, & \textrm{ if } k=-1, \\ 0, & \textrm{ if } k\le -2.  \end{array}\right.$$
Thus we get $R_1(n)=0$ for $n\le 1$. For $n\ge 2$, we have 
\begin{align*}
R_1(n)=&q^{-12s+3}\omega_\tau(p^2)\\
&\cdot\left(W_{f_s}(h(1,p^{n-1}))+\sum_{m=1}^{n-2}(1-q^{-1})q^mW_{f_s}( h(p^{m},p^{n-m-1}))-q^{-1}q^{n-1}W_{f_s}(h(p^{(n-1)},1)) \right).\\
=&q^{-12s+3}\omega_\tau(p^2)\\
&\cdot\left(W_{f_s}(h(1,p^{n-1}))+\sum_{m=1}^{n-1}(1-q^{-1})q^mW_{f_s}( h(p^{m},p^{n-m-1}))-q^{n-1}W_{f_s}(h(p^{(n-1)},1)) \right)\\
=&q^{-12s+3}\omega_\tau(p)^2I(n-1)-q^{-12s+3+n-1}\omega_\tau(p)^2W_{f_s}(h(p^{n-1},1)),
\end{align*}
where in the last equation, we used the formula in the computation of $I(n)$. Since $h(p^{n-1},1)$ is in the center of $M'$, we have $W_{f_s}(h(p^{n-1},1))=q^{-6s(n-1)}\omega_\tau(p)^{n-1}$. The result follows.
\end{proof}

We next consider \begin{align*}
R_2(n)&=\int_{|r_4|>1} \int_{F} W_{f_s}(h(1,p^n)w_\beta w_\alpha w_\beta \bx_{\alpha+\beta}(p^{-1})\bx_\beta(r_4)\bx_{3\alpha+2\beta}(r_5))dr_4dr_5 .
\end{align*}
Conjugating $w_\beta$ to the right side and using the Iwasawa decomposition of $w_\beta \bx_\beta(r_4)$, we can get
\begin{align*}
R_2(n)=\int_F \int_{|r_4|>1}W_{f_s}(h(1,p^n)w_\beta w_\alpha \bx_\alpha (p^{-1})\bx_{3\alpha+\beta}(r_5)\bx_\beta(r_4^{-1})h(r_4^{-1},r_4))dr_4dr_5.
\end{align*}
From the commutator relation, we have
$$\bx_\alpha(p^{-1})\bx_\beta(r_4^{-1})=\bx_\beta(r_4^{-1})\bx_\alpha(p^{-1})\bx_{2\alpha+\beta}(p^{-2}r_4^{-1})u,$$
for some $u$ in the group generated by roots subgroups of $\alpha+\beta,3\alpha+\beta,3\alpha+2\beta$. Like in the computation of $R_1(n)$, we have 
\begin{align*}
R_2(n)&=\int_F\int_{|r_4|>1}W_{f_s}(h(1,p^n)w_\beta w_\alpha \bx_\alpha(p^{-1})\bx_{3\alpha+\beta}(r_5)h(r_4^{-1},r_4))\psi(-2p^{n-2}r_4^{-1})dr_4dr_5\\
&=\int_F \int_{|r_4|>1}W_{f_s}(h(1,p^n)h(r_4^{-1},1)w_\beta \bx_\beta(r_5r_4^{-1})w_\alpha \bx_\alpha(p^{-1}r_4^{-1}))\psi(-2p^{n-2}r_4^{-1})dr_4dr_5\\
&=\int_F \int_{|r_4|>1}|r_4|W_{f_s}(h(1,p^n)h(r_4^{-1},1)w_\beta \bx_\beta(r))\psi(-2p^{n-2}r_4^{-1})dr_4dr\\
&=I(n)\int_{|r_4|>1}|r_4|^{-6s+1}\omega_\tau(r_4^{-1})\psi(-2p^{n-2}r_4^{-1})dr_4\\
&=I(n)\sum_{m= 1}^\infty q^{(-6s+2)m}\omega_\tau(p)^{m}\int_{\fo^\times}\psi(-2p^{m+n-2}u)du.
\end{align*}
\begin{lem}
We have 
$$R_2(n)=\left\{\begin{array}{lll}I(0)q^{-6s+2}\omega_\tau(p)\left(-q^{-1}+(1-q^{-1})\frac{q^{-6s+2}\omega_\tau(p)}{1-q^{-6s+2}\omega_\tau(p)} \right),& n=0, \\ I(n)(1-q^{-1})\frac{q^{-6s+2}\omega_\tau(p)}{1-q^{-6s+2}\omega_\tau(p)},& n\ge 1 \end{array}\right.$$
\end{lem}
\begin{proof}
If $n\ge 1$, then $\int_{\fo^\times}\psi(p^{m+n-2}u)du=(1-q^{-1}) $ for $m\ge 1$. Thus, we have 
\begin{align*}
R_2(n)&=I(n)\sum_{m= 1}^\infty q^{(-6s+2)m}\omega_\tau(p)^{m}(1-q^{-1})\\
&=I(n)(1-q^{-1})\frac{q^{-6s+2}\omega_\tau(p)}{1-q^{-6s+2}\omega_\tau(p)}.
\end{align*}
If $n=0$, then $\int_{\fo^\times}\psi(p^{m+n-2}u)du=(1-q^{-1}) $ for $m\ge 2$, and $\int_{\fo^\times}\psi(p^{m+n-2}u)du=-q^{-1}$ for $m=1$. Thus, we have 
\begin{align*}
R_2(0)&=I(0)(-q^{-1}q^{-6s+2}\omega_\tau(p)+(1-q^{-1})\sum_{m=2}^\infty q^{(-6s+2)m}\omega_\tau(p)^{m})\\
&=I(0)q^{-6s+2}\omega_\tau(p)\left(-q^{-1}+(1-q^{-1})\frac{q^{-6s+2}\omega_\tau(p)}{1-q^{-6s+2}\omega_\tau(p)} \right).
\end{align*}
The completes the proof of the lemma.
\end{proof}
Combining the above results, we get the following
\begin{lem}
We have 
\begin{align*}
R(n)=\left\{\begin{array}{lll}-I(0)q^{-6s+1}\omega_\tau(p)\frac{1-q^{-6s+3}\omega_\tau(p)}{1-q^{-6s+2}\omega_\tau(p)}, & n=0,\\
 I(1)(1-q^{-1})\frac{q^{-6s+2}\omega_\tau(p)}{1-q^{-6s+2}\omega_\tau(p)}, & n=1,\\
q^{-12s+3}\omega_\tau(p)^2I(n-1)-q^{-6s(n+1)+n+2}\omega_\tau(p)^{n+1} &\\
\qquad \qquad \qquad +I(n)(1-q^{-1})\frac{q^{-6s+2}\omega_\tau(p)}{1-q^{-6s+2}\omega_\tau(p)}, & n\ge 2,
 \end{array}\right.
\end{align*}
and 
\begin{align*}
J(n)&=J_1(n)-R(n)\\
&=\left\{\begin{array}{lll}1+Y, & n=0\\
I(1), & n=1,\\
I(n)-q^{-1}Y^2 I(n-1)+q^{-n}Y^{n+1}, & n\ge 2.
  \end{array} \right.
\end{align*}
where $Y=q^{-6s+2}\omega_\tau(p)$
\end{lem}
By the main result of \cite{BFH}, we have 
$$\wt W(t(p^n))=\frac{\mu_\psi(p^{n})q^{-n}}{a-a^{-1}}\left( (1-\chi(p)q^{-1/2}a^{-1})a^{n+1}-(1-\chi(p)q^{-1/2}a)a^{-(n+1)}\right),$$
where $\chi(p)=(p,p)_F=(p,-1)_F.$ Note that the notation $\gamma(a)$ in \cite{BFH} is our $\mu_\psi(a)^{-1}$. Note that $\mu_\psi(p^n)\mu_\psi(p^n)=(p^n,p^n)_F=\chi(p)^n$.
Thus 
\begin{align*}
I(\wt W,W_{f_s},\phi)&=\sum_{n\ge 0}\frac{q^{3n/2}\chi(p)^n}{a-a^{-1}}\left( (1-\chi(p)q^{-1/2}a^{-1})a^{n+1}-(1-\chi(p)q^{-1/2}a)a^{-(n+1)}\right) J(n).
\end{align*}

Plugging the formula $J(n)$ into the above equation, we can get that 
\begin{align*}
    I(\wt W,W_f,\phi)
    =&\frac{(1-b_1q^{-1}X)(1-b_2q^{-1}X)(1-b_1b_2q^{-1}X^2)(1-b_1^2b_2q^{-1}X^3)(1-b_1b_2^2q^{-1}X^3)}{(1-\chi(p)a^{-1}b_1b_2q^{-1/2}X^2)(1-\chi(p)ab_1b_2q^{-1/2}X^2)}\\
    &\cdot \frac{1}{\prod_{i=1}^2(1-\chi(p)a^{-1}b_iq^{-1/2}X)\prod_{i=1}^2(1-\chi(p)ab_iq^{-1/2}X)}\\
    =&\frac{L(3s-1,\wt\pi\times (\chi\otimes\tau))L(6s-5/2,\wt\pi\otimes (\chi\otimes\omega_\tau))}{L(3s-1/2,\tau)L(6s-2,\omega_\tau)L(9s-7/2,\tau\otimes\omega_\tau)}.
\end{align*}
Here $$L(s,\wt \pi\otimes (\chi\otimes \tau))=\frac{1}{(1-a\chi(p)b_1b_2q^{-s})((1-a^{-1}\chi(p)b_1b_2q^{-s}))}$$
is the $L$ function of $\wt\pi$ twisted by the character $\chi\otimes \omega_\tau $, and 
$$L(s,\wt \pi\times(\chi\otimes \tau))=\frac{1}{\prod_{i=1}^2(1-\chi(p)a^{-1}b_iq^{-s})\prod_{i=1}^2(1-\chi(p)ab_iq^{-s})}$$
is the Rankin-Selberg $L$-function of $\wt\pi$ twisted by $\chi\otimes \tau$.

We record the above calculation in the following 
\begin{prop}\label{prop: unramified calculation}
Let $\wt W\in \CW(\wt \pi,\psi)$ be the normalized unramified Whittaker function, $f_s$ be the normalized unramified section in $I(s,\tau)$ and $\phi\in \CS(F)$ is the characteristic function of $\fo$, we have 
$$ I(\wt W,W_{f_s},\phi)=\frac{L(3s-1,\wt\pi\times (\chi\otimes\tau))L(6s-5/2,\wt\pi\otimes (\chi\otimes\omega_\tau))}{L(3s-1/2,\tau)L(6s-2,\omega_\tau)L(9s-7/2,\tau\otimes\omega_\tau)}. $$
\end{prop}

\section{Some local theory}
In this section, let $F$ be a local field, which can be archimedean or non-archimedean. If $F$ is non-archimedean, let $\fo$ be the ring of integers of $F$, $ p$ be a uniformizer of $\fo$ and $q=\fo/(p)$. Let $\wt\pi$ be an irreducible genuine generic representation of $\wt \SL_2(F)$, $\tau$ be an irreducible generic representation of $\GL_2(F)$. Let $\psi$ be a nontrivial additive character of $F$. 

\begin{lem}
Let $\wt W\in \CW(\wt\pi,\psi), f_s\in I(s,\tau),\phi\in \CS(F)$, then the integral $I(\wt W,W_{f_s},\phi)$ converges absolutely for $\Re(s)$ large and has a meromorphic continuation to the whole $s$-plane. Moreover, if $F$ is a $p$-adic field, then $I(\wt W,W_{f_s},\phi)$ is a rational function in $q^{-s}$.
\end{lem}
The proof is similar to \cite[Lemma 4.2-4.7]{Gi93} and \cite[Lemma 3.10, Lemma 3.3]{Gi95}. We omit the details. 

\begin{lem}
 Let $s_0\in \BC$. Then there exists $\wt W\in \CW(\wt\pi,\psi),f_{s_0}\in I(s_0,\tau),\phi\in \CS(F)$ such that $I(\wt W,W_{f_{s_0}},\phi)\ne 0$.
\end{lem}
\begin{proof}
The proof is similar to the proof of \cite[Lemma 4.4,4.7]{Gi93}, \cite[Proposition 3.4]{Gi95}. We omit the details.
\end{proof}

\section{Nonvanishing of certain periods on \texorpdfstring{$G_2$}{}}

\subsection{Poles of Eisenstein series on \texorpdfstring{$\RG_2$}{}}\label{sec: poles of Eisenstein series}
Let $\tau$ be a cuspidal unitary representation of $\GL_2(\BA)\cong M'(\BA)$. Let $K$ be a maximal compact subgroup of $\RG_2(\BA)$. Given a $K\cap \GL_2(\BA)$-finite cusp form $f$ in $\tau$, we can extend $f$ to a function $\wt f:\RG_2(\BA)\ra \BC$ as in \cite[\S 2]{Sh}. We then define $$\Phi_{\wt f,s}(g)=\wt f(g) \delta_{P'}(m')^{s/3+1/2},$$
for $g=v'm'k$ with $v'\in V'(\BA),m'\in M'(\BA),k\in K$. Then $\Phi_{\wt f,s}$ is well-defined and $\Phi_{\wt f,s}\in I(\frac{s}{3}+\frac{1}{2},\tau)$. Then we can consider the Eisenstein series
$$E(s,\wt f,g)=\sum_{P'(F)\backslash \RG_2(F)}\Phi_{\wt f,s}(\gamma g).$$
\begin{prop}
The Eisenstein series $E(s,\wt f, g)$ has a pole on the half plane $\Re(s)>0$ if and only if $s=\frac{1}{2}, \omega_\tau=1$ and $L(\frac{1}{2},\tau)\ne 0$.
\end{prop}
For a proof of the above proposition, see \cite[\S 1]{Za} or \cite[\S 5]{Kim}. If $\omega_\tau=1$ and $L(\frac{1}{2},\tau)\ne 0$, denote by $\CR(\frac{1}{2},\tau)$ the space generated by the residues of Eisenstein series $E(s,\wt f,g)$ defined as above. Note that an element $ R\in \CR(\frac{1}{2},\tau)$ is an automorphic form on $G_2(\BA) $.

\subsection{On the Shimura-Waldspurger lift} Let $\wt\pi$ be a genuine cuspidal automorphic representation of $\wt {\SL}_2(\BA)$. Let $Wd_\psi(\wt\pi)$ be the Shimura-Waldspurger lift of $\wt\pi$. Then $Wd_\psi(\wt \pi)$ is a cuspidal representation of $\RP\GL_2(\BA)$. A cuspidal automorphic representation $ \tau$ is in the image of $Wd_\psi$ if and only if $L(\frac{1}{2},\tau)\ne 0$. Moreover, the correspondence $\wt \pi\mapsto Wd_\psi(\wt\pi)$ respects the Rankin-Selberg $L$-functions. For these assertions, see \cite{Wald} or \cite{G}.

\subsection{A period on \texorpdfstring{$G_2$}{}}
\begin{thm}\label{thm: nonvanishing of a period}
Let $\wt \pi$ be a genuine cuspidal automorphic representation of $\wt \SL_2(\BA)$ and $\tau$ be a unitary cuspidal automorphic representation of $\GL_2(\BA)$. Assume that $\omega_\tau=1$ and $L(\frac{1}{2},\tau)\ne 0$. In particular, $\tau$ can be viewed as a cuspidal automorphic representation of $\RP\GL_2(\BA)$. If $Wd_\psi(\wt \pi)=\chi\otimes\tau,$
   then there exists $\wt \varphi \in V_{\wt \pi}, \phi\in \CS(\BA), R\in \CS(\frac{1}{2},\tau)$ such that the period 
    $$\CP(\wt \varphi, \wt\theta_\phi,R)=\int_{\SL_2(F)\backslash \SL_2(\BA)}\int_{V(F)\backslash V(\BA)}\wt \varphi(g)\wt \theta_\phi(vg)R(vg)dvdg $$
    is non-vanishing.
\end{thm}
\begin{proof}
 For $\wt \varphi\in V_\pi, \phi\in \CS(\BA)$ and a good section $\Phi_{\wt f, s}$ as in \S\ref{sec: poles of Eisenstein series}, by Theorem \ref{thm: eulerian} and Proposition \ref{prop: unramified calculation}, we have
\begin{align*}%\label{eq: eulerian}
    I(\wt \varphi,\phi,\wt f,s)&=\int_{\SL_2(F)\backslash \SL_2(\BA)}\int_{V(F)\backslash V(\BA)} \wt \varphi(g)\wt\theta_{\phi}(vg)E(vg,\Phi_{\wt f,s})dvdg\\
    &=\int_{N_{\SL_2}(\BA)\backslash \SL_2(\BA)}\int_{U_{\alpha+\beta}(\BA)\backslash V(\BA)}W_{\wt \varphi}(g)\omega_\psi(vg)\phi(1)W_{\Phi_{\wt f,s}}(\gamma vg)dvdg \nonumber\\
    &=I_S \cdot \frac{L^S(s+\frac{1}{2},\wt \pi\times (\chi\otimes\tau))L^S(2s+\frac{1}{2},\wt \pi\otimes (\chi\otimes\omega_\tau))}{L^S(s+1,\tau)L^S(2s+1,\omega_\tau)L^S(3s+1,\tau\otimes \omega_\tau)}. \nonumber
\end{align*}
Here $S$ is a finite set of places of $F$ such that for $v\notin S$, $ \pi_v,\tau_v$ are unramified, and $ I_S$ is the product of the local zeta integrals over all places $v\in S$ and $L^S$ denotes the partial $L$-function which is the product of all local $L$-function as the place $v$ runs over $v\notin S$. Note that $\tau\cong \tau^\vee$ since $\omega_\tau=1$. Suppose that $Wd_\psi(\wt \pi)=\chi\otimes \tau=\chi\otimes\tau^\vee$, then $L^S(s+1/2, \wt \pi \times (\chi\otimes\tau)) $ has a pole at $s=1/2$. Note that at $s=\frac{1}{2}$, $L^{S}(2s+1/2,\wt \pi\otimes (\chi\otimes\omega_\tau))$ is holomorphic and nonzero, while $L^S(s+1,\tau)L^S(2s+1,\omega_\tau)L^S(3s+1,\tau\otimes\omega_\tau)$ is holomorphic. Moreover, $I_S$ can be chosen to be nonzero. Thus we get that $I(\wt \varphi,\phi, \wt f,s)$
has a pole at $s=1/2$, which means that there exists a residue $R(g,\wt f )$ of $E(s,\wt f,g)$ such that 
$$\CP(\wt \varphi,\theta_\phi,R)=\int_{\SL_2(F)\backslash \SL_2(\BA)}\int_{V(F)\backslash V(\BA)}\wt \varphi(g)\wt \theta_\phi(vg)R(vg,\wt f)dvdg\ne 0.$$
This completes the proof.
\end{proof}

\begin{rmk}{\rm 
For an $L^2$-automorphic form $\eta\in L^2(G_2(F)\backslash G_2(\BA))$, one can form the period $$\eta_{\wt\phi,\wt\theta_\phi}(g)=\int_{\SL_2(F)\backslash \SL_2(\BA)}\int_{V(F)\backslash V(\BA)}\wt \varphi(h)\wt\theta_\phi(vh)\eta(vhg)dvdh,$$
for a genuine cusp form $\wt \phi$ of $\wt{\SL}_2(\BA)$ and $\phi\in \CS(\BA)$. Theorem \ref{thm: nonvanishing of a period} says that if $\eta$ varies in $ \CS(\frac{1}{2},\tau)$, then under the condition $Wd_\psi(\wt \pi)=\chi\otimes \tau$, the period $\eta_{\wt \varphi,\wt \theta_\phi} $ is non-vanishing for certain $\wt\varphi\in V_{\wt\pi}$ and $\phi\in \CS(\BA)$. For general $\eta$, one can ask under what conditions the period $\eta_{\wt\varphi,\wt\theta_\phi}$ is not identically zero as $\wt\varphi$ varies in $V_{\wt\pi}$ and $\phi\in \CS(\BA)$. In the classical group case, this is the global Gan-Gross-Prasad conjecture for Fourier-Jacobi case, see 
\cite{GGP}. It is natural to ask if it is possible to extend the GGP-conjecture to the $G_2$-case.
}
\end{rmk}

\end{document}